\documentclass[a4paper,12pt,reqno]{amsart}

\usepackage{amsmath}
\usepackage{amssymb}
\usepackage{amsfonts}
\usepackage{graphicx}
\usepackage[colorlinks]{hyperref}
\renewcommand\eqref[1]{(\ref{#1})} 

\graphicspath{ {images/} }
\title[Subelliptic geometric Hardy type inequalities]{Subelliptic geometric Hardy type inequalities on half-spaces and convex domains}
\author{Michael Ruzhansky}
\address{Imperial College London, London, UK}
\email{m.ruzhansky@imperial.ac.uk}

\author{Bolys Sabitbek}
\address{Institute of Mathematics and Mathematical Modeling 
	\endgraf 
	and Al-Farabi Kazakh National University, Almaty, Kazakhstan}

\email{b.sabitbek@math.kz}

\author{Durvudkhan Suragan}
\address{Department of Mathematics
	\endgraf
	School of Science and Technology, Nazarbayev University
	\endgraf
	53 Kabanbay Batyr Ave, Astana 010000
	\endgraf
	Kazakhstan}
\email{durvudkhan.suragan@nu.edu.kz}
\thanks{The first author was supported by the EPSRC Grant 
	EP/R003025/1 and by 
	the Leverhulme Research Grant RPG-2017-151. The second author was supported by the MESRK grant BR05236656. The third author was partially supported by the NU SPG and the MESRK grant AP05130981.}


\newtheoremstyle{theorem}
{10pt}          
{10pt}  
{\sl}  
{\parindent}     
{\bf}  
{. }    
{ }    
{}     
\theoremstyle{theorem}

\numberwithin{equation}{section}
\theoremstyle{plain}
\newtheorem{thm}{Theorem}[section]

\newtheorem{cor}[thm]{Corollary}

\theoremstyle{definition}

\newtheorem{rem}[thm]{Remark}

\newcommand{\G}{\mathbb{G}}

\newtheoremstyle{defi}
{10pt}          
{10pt}  
{\rm}  
{\parindent}     
{\bf}  
{. }    
{ }    
{}     
\theoremstyle{defi}

\begin{document}
		\begin{abstract}
		In this paper we present $L^2$ and $L^p$ versions of the geometric Hardy inequalities in  half-spaces and convex domains on stratified (Lie) groups. As a consequence, we obtain the geometric uncertainty principles. We give examples of the obtained results for the Heisenberg and the Engel groups.
	\end{abstract}
	
	\maketitle

	\section{Introduction}
	In the Euclidean setting, a geometric Hardy inequality in a (Euclidean) convex domain $\Omega$ has the following form  
	\begin{equation}
		\int_{\Omega} |\nabla u|^2 dx \geq \frac{1}{4} \int_{\Omega} \frac{|u|^2}{dist(x,\partial \Omega)^2} dx, 
	\end{equation}
	for $u \in C_0^{\infty}(\Omega)$ with the sharp constant $1/4$. There is a number of studies related to this subject, see e.g. \cite{Ancona}, \cite{DAmbrozio}, \cite{Avka_Lap}, \cite{Avk_Wirth}, \cite{Davies} and \cite{Opic-Kuf}.

	In the case of the Heisenberg group $\mathbb{H}$, Luan and Yang \cite{Luan_Yang} obtained the following Hardy inequality on the half space $\mathbb{H}^+ := \{(x_1,x_2,x_3)\in \mathbb{H} \, | \, \, x_3>0 \}$   for $u \in C_0^{\infty}(\mathbb{H}^+)$
	\begin{equation}
		\int_{\mathbb{H}^+} |\nabla_{\mathbb{H}} u|^2 dx \geq  \int_{\mathbb{H}^+} \frac{|x_1|^2+|x_2|^2}{x_3^2}|u|^2 dx.  
	\end{equation} 
	Moreover, the geometric $L^p$-Hardy inequalities for the sub-Laplacian on the convex domain in the Heisenberg group was obtained by Larson \cite{Larson} which also generalises the previous result of \cite{Luan_Yang}. In this note by using the approach in \cite{Larson} we obtain the geometric Hardy type inequalities on the half-spaces and the convex domains on general stratified groups, so our results extend known results of Abelian (Euclidean) and Heisenberg groups.
	
	Thus, the main aim of this paper is to prove the geometric Hardy type inequalities on general stratified groups. As consequences, the geometric uncertainty principles are obtained. We also demonstrate the obtained results for some concrete examples of step 2 and step 3 stratified groups. In Section \ref{half} we present $L^2$ and $L^p$ versions of the subelliptic geometric Hardy type inequalities on the half-space. In Section \ref{convex}, we show subelliptic $L^2$ and $L^p$ versions of the geometric Hardy type inequalities on the convex domains.  
	\subsection{Preliminaries}    
	Let $\mathbb{G}=(\mathbb{R}^n,\circ,\delta_{\lambda})$ be a stratified Lie group (or a homogeneous Carnot group), with dilation structure $\delta_{\lambda}$ and Jacobian generators $X_{1},\ldots,X_{N}$, so that $N$ is the dimension of the first stratum of $\mathbb{G}$. We denote by $Q$ the homogeneous dimension of $\G$.  We refer to \cite{FS}, or to the recent books  \cite{BLU} and \cite{FR} for extensive discussions of stratified Lie groups and their properties.
		
	The sub-Laplacian on $\mathbb{G}$ is given by
	\begin{equation}\label{sublap}
	\mathcal{L}=\sum_{k=1}^{N}X_{k}^{2}.
	\end{equation}
	We also recall that the standard Lebesque measure $dx$ on $\mathbb R^{n}$ is the Haar measure for $\mathbb{G}$ (see, e.g. \cite[Proposition 1.6.6]{FR}).
	Each left invariant vector field $X_{k}$ has an explicit form and satisfies the divergence theorem,
	see e.g. \cite{FR} for the derivation of the exact formula: more precisely, we can formulate 
	\begin{equation}\label{Xk0}
	X_{k}=\frac{\partial}{\partial x'_{k}}+
	\sum_{l=2}^{r}\sum_{m=1}^{N_{l}}a_{k,m}^{(l)}(x',...,x^{(l-1)})
	\frac{\partial}{\partial x_{m}^{(l)}},
	\end{equation}
	with $x=(x',x^{(2)},\ldots,x^{(r)})$, where $r$ is the step of $\G$ and
	$x^{(l)}=(x^{(l)}_1,\ldots,x^{(l)}_{N_l})$ are the variables in the $l^{th}$ stratum,
	see also \cite[Section 3.1.5]{FR} for a general presentation.
	The horizontal gradient is given by
	$$\nabla_{\G}:=(X_{1},\ldots, X_{N}),$$
	and the horizontal divergence is defined by
	$${\rm div}_{\G} v:=\nabla_{\G}\cdot v.$$

	We now recall the divergence formula in the form of \cite[Proposition 3.1]{RS17b}. 
Let $f_{k}\in C^{1}(\Omega)\bigcap C(\overline{\Omega}),\,k=1,\ldots,N$.
Then for each $k=1,\ldots,N,$ we have 
\begin{equation}\label{EQ:S1}
\int_{\Omega}X_{k}f_{k}dz=
\int_{\partial\Omega}f_{k} \langle X_{k},dz\rangle.
\end{equation}
Consequently, we also have 
\begin{equation}\label{EQ:S2}
\int_{\Omega}\sum_{k=1}^{N}X_{k}f_{k}dz=
\int_{\partial\Omega}
\sum_{k=1}^{N} f_{k}\langle X_{k},dz\rangle.
\end{equation}


\section{Hardy type inequalities on half-space}\label{half}
\subsection{$L^2$-Hardy inequality on the half-space of $\G$}
In this section we present the geometric $L^2$-Hardy inequality on the half-space of $\G$. We define the half-space as follows
\begin{equation*}
	\G^+ := \{ x \in \G: \langle x, \nu \rangle > d  \},
\end{equation*}
where $\nu:=(\nu_1,\ldots,\nu_r)$ with $\nu_j \in \mathbb{R}^{N_j},\,\, j=1,\ldots,r,$ is the Riemannian outer unit normal to $\partial \G^+$ (see \cite{Garofalo}) and $d \in \mathbb{R}$. The Euclidean distance to the boundary $\partial \G^+$ is denoted by $dist(x,\partial \G^+)$ and defined as follows
\begin{equation}
	dist(x,\partial \G^+)= \langle x, \nu \rangle - d.
\end{equation} 
Moreover, there is an angle function on $\partial \G^+$ which is defined by Garofalo in \cite{Garofalo} as
\begin{equation}\label{2.2}
\mathcal{W}(x) = \sqrt{\sum_{i=1}^{N}\langle X_i(x), \nu\rangle^2}. 
\end{equation}
\begin{thm}\label{Hardy_hs}
	Let $\G^+$ be a half-space of a stratified group $\G$. Then for all $\beta \in \mathbb{R}$ we have
	\begin{align}\label{Hardy_1}
			\int_{\G^+} |\nabla_{\G} u|^2 dx \geq  & C_1(\beta)\int_{\G^+} \frac{\mathcal{W}(x)^2}{dist(x,\partial \G^+)^2} |u|^2 dx \\
			&+ \beta \int_{\G^+}\sum_{i=1}^{N} \frac{X_i\langle X_i(x),\nu \rangle}{dist(x,\partial \G^+)} |u|^2 dx, \nonumber
	\end{align}
	for all $u \in C^{\infty}_0(\G^+)$ and where $C_1(\beta):=-(\beta^2+ \beta)$.
\end{thm}
\begin{rem}
	If $\G$ has step $r=2$, then for $i=1,\ldots,N$ we have the following left-invariant vector fields 
	\begin{equation}
		X_i = \frac{\partial}{\partial x'_i} + \sum_{s=1}^{N_2}\sum_{m=1}^{N} a_{m,i}^{s} x'_{m}\frac{\partial}{\partial x''_s}, 
	\end{equation}
	where $a_{m,i}^s$ are the group constants (see, e.g. \cite[Formula (2.14)]{DGN} for the definition). Also we have $x:=(x',x'')$ with $x'= (x'_1,\ldots,x_N')$, $x''=(x''_1,\ldots,x''_{N_2})$, and also $\nu:=(\nu',\nu'')$ with $\nu'= (\nu'_1,\ldots,\nu_N')$ and $\nu''=(\nu''_1,\ldots,\nu''_{N_2})$.
	\begin{cor}
			Let $\G^+$ be a half-space of a stratified group $\G$ of step $r=2$. For all $\beta \in \mathbb{R}$ and $u \in C^{\infty}_0(\G^+)$ we have
		\begin{align}\label{Hardy_122}
		\int_{\G^+} |\nabla_{\G} u|^2 dx \geq  & C_1(\beta)\int_{\G^+} \frac{\mathcal{W}(x)^2}{dist(x,\partial \G^+)^2} |u|^2 dx \\
		+& K(a,\nu,\beta) \int_{\G^+} \frac{|u|^2}{dist(x,\partial \G^+)}  dx, \nonumber
		\end{align}
		 where $C_1(\beta):=-(\beta^2+ \beta)$ and $K(a,\nu,\beta):= \beta \sum_{s=1}^{N_2}\sum_{i=1}^{N}a_{i,i}^s\nu_{s}''$.
	\end{cor} 
\end{rem}
\begin{proof}[Proof of Theorem \ref{Hardy_hs}.]
To prove inequality \eqref{Hardy_1} we use the method of factorization. Thus, for any $W:=(W_1,\ldots,W_N),\,\, W_i \in C^1(\G^+)$ real-valued, which will be chosen later, by a simple computation we have 
	\begin{align*}
		0\leq \int_{\G^+} |\nabla_{\G} u + \beta W u|^2 dx &= \int_{\G^+} | (X_1u,\ldots,X_N u) + \beta (W_1,\ldots, W_N )u|^2 dx \\
		& = \int_{\G^+} | (X_1u+\beta W_1 u,\ldots,X_N u + \beta W_N u) |^2 dx \\
		& = \int_{\G^+} \sum_{i=1}^{N} |X_i u+ \beta W_i u|^2 dx \\
		& = \int_{\G^+}  \sum_{i=1}^{N} \left[ |X_i u|^2 +2 {\rm Re} \beta W_i u X_iu + \beta^2 W_i^2 |u|^2 \right] dx \\ 
		& = \int_{\G^+}  \sum_{i=1}^{N} \left[ |X_i u|^2 + \beta W_i X_i|u|^2 + \beta^2 W_i^2 |u|^2 \right] dx \\ 
		& = \int_{\G^+}  \sum_{i=1}^{N} \left[ |X_i u|^2 - \beta (X_iW_i) |u|^2 + \beta^2 W_i^2 |u|^2 \right] dx. 
	\end{align*}
	From the above expression we get the inequality
	\begin{equation}\label{eq1}
		 \int_{\G^+} |\nabla_{\G} u|^2 dx \geq \int_{\G^+} \sum_{i=1}^{N} \left[  (\beta (X_iW_i)  - \beta^2 W_i^2 )|u|^2 \right] dx.
	\end{equation}
	Let us now take $W_i$ in the form
	\begin{equation}\label{W}
		W_i (x) = \frac{\langle X_i(x),\nu \rangle}{dist(x,\partial \G^+)} = \frac{\langle X_i(x),\nu \rangle}{\langle x, \nu \rangle - d},
	\end{equation}
	where 
	\begin{equation*}
		X_i(x) = (\overset{i}{\overbrace{(0,\ldots,1}},\ldots,0,a_{i,1}^{(2)}(x'),\ldots,a_{i,N_r}^{(r)}(x',x^{(2)},\ldots,x^{(r-1)})),
	\end{equation*}
	and
	\begin{equation*}
		\nu = (\nu_1,\nu_2,\ldots, \nu_r), \,\, \nu_j \in \mathbb{R}^{N_j}.
	\end{equation*}
	Now $W_i(x)$ can be written as 
	\begin{equation*}
		W_i(x) = \frac{\nu_{1,i}+ \sum_{l=2}^{r}\sum_{m=1}^{N_l}a_{i,m}^{(l)}(x',\ldots,x^{(l-1)})\nu_{l,m}}{\sum_{l=1}^{r} x^{(l)}\cdot \nu_{l}-d}.
	\end{equation*}
	By a direct computation we have
	\begin{align}\label{1}
		X_i W_i(x) &= \frac{X_i\langle X_i(x),\nu \rangle dist(x,\partial \G^+) - \langle X_i(x),\nu \rangle X_i( dist(x,\partial \G^+))}{dist(x,\partial \G^+)^2} \nonumber \\ &=\frac{X_i\langle X_i(x),\nu \rangle}{dist(x,\partial \G^+)}- \frac{\langle X_i(x),\nu \rangle^2}{dist(x,\partial \G^+)^2}, 
	\end{align}
	where 
	\begin{align*}
		X_i( dist(x,\partial \G^+)) &= X_i\left( \sum_{k=1}^{N} x'_k\nu_{1,k} + \sum_{l=2}^{r}\sum_{m=1}^{N_l} x_m^{(l)}\nu_{l,m} -d\right)\\
		 & = \nu_{1,i} + \sum_{l=2}^{r}\sum_{m=1}^{N_l}a_{i,m}^{(l)}(x',\ldots,x^{(l-1)})\nu_{l,m} \\
		 &= \langle X_i(x), \nu \rangle.
	\end{align*}
	Inserting the expression \eqref{1} in \eqref{eq1} we get
	\begin{equation*}
	\int_{\G^+} |\nabla_{\G} u|^2 dx \geq   -(\beta^2+ \beta)\int_{\G^+} \sum_{i=1}^{N} \frac{\langle X_i(x),\nu \rangle^2}{dist(x,\partial \G^+)^2} |u|^2 dx + \beta \int_{\G^+}\sum_{i=1}^{N} \frac{X_i\langle X_i(x),\nu \rangle}{dist(x,\partial \G^+)} |u|^2 dx.
	\end{equation*}
	The proof of Theorem \ref{Hardy_hs} is finished.
\end{proof}
As consequences of Theorem \ref{Hardy_hs}, we have the geometric Hardy inequalities on the half-space without an angle function, which seems an interesting new result on $\G$.
\begin{cor}\label{cor0}
	Let $\G^+$ be a half-space of a stratified group $\G$. Then we have
	\begin{equation}\label{Hardy_3}
	\int_{\G^+} |\nabla_{\G} u|^2 dx \geq \frac{1}{4} \int_{\G^+} \frac{|u|^2}{dist(x,\partial \G^+)^2}dx,
	\end{equation}
	for all $u \in C^{\infty}_0(\G^+)$.
\end{cor}
\begin{proof}[Proof of Corollary \ref{cor0}] 
	Let $x:=(x',x^{(2)},\ldots,x^{(r)}) \in \G$ with $x'=(x'_1,\ldots,x'_N)$ and $x^{(j)} \in \mathbb{R}^{N_j}, \,\, j=2,\ldots,r$. By taking $\nu :=(\nu',0,\ldots,0)$ with $\nu'= (\nu'_1,\ldots,\nu'_N),$ we have that
	\begin{equation*}
		X_i(x) = (\overset{i}{\overbrace{0,\ldots,1},\ldots,0}, a_{i,1}^{(2)}(x'),\ldots,a_{i,N_r}^{(r)}(x',x^{(2)},\ldots,x^{(r-1)}))
	\end{equation*} we have 
\begin{align*}
\sum_{i=1}^{N}	\langle X_i(x), \nu \rangle^2 = \sum_{i=1}^{N}(\nu_i')^2=|\nu'|^2=1,
\end{align*}
and 
\begin{equation*}
	X_i \langle X_i(x), \nu \rangle = X_i \nu'_i = 0.
\end{equation*}
Inserting the above expressions in inequality \eqref{Hardy_1} we arrive at 
\begin{equation*}
		\int_{\G^+} |\nabla_{\G} u|^2 dx \geq  -(\beta^2+ \beta)\int_{\G^+} \frac{|u|^2}{dist(x,\partial \G^+)^2} dx.
\end{equation*}
For optimisation we differentiate the right-hand side of integral with respect to $\beta$, then we have
\begin{equation*}
-2\beta-1 =0,
\end{equation*}
which implies 
\begin{equation*}
\beta = -\frac{1}{2}.
\end{equation*}
This completes the proof.
\end{proof}
We also have the geometric uncertainty principle on the half-space of $\G^+$. 
\begin{cor}\label{cor2}
	Let $\G^+$ be a half-space of a stratified group $\G$. Then we have
	\begin{equation}
	\left( \int_{\G^+} |\nabla_{\G} u|^2 dx\right)^{\frac{1}{2}} \left( \int_{\G^+} dist(x,\partial \G^+)^2 |u|^2 dx \right)^{\frac{1}{2}}
	\geq \frac{1}{2} \int_{\G^+} |u|^2 dx
	\end{equation}
	for all $u \in C^{\infty}_0(\G^+)$. 
\end{cor}
\begin{proof}[Proof of Corollary \ref{cor2}]
	By using \eqref{Hardy_3} and the Cauchy-Schwarz inequality we get 
	\begin{align*}
	&\int_{\G^+} |\nabla_{\G} u|^2 dx  \int_{\G^+} dist(x,\partial \G^+)^2 |u|^2 dx \\
	&\geq \frac{1}{4}\int_{\G^+} \frac{1}{dist(x,\partial \G^+)^2} |u|^2 dx \int_{\G^+} dist(x,\partial \G^+)^2 |u|^2 dx \\
	& \geq \frac{1}{4} \left(\int_{\G^+} |u|^2 dx\right)^2.
	\end{align*}
\end{proof}
To demonstrate our general result in a particular case, here we consider the Heisenberg group, which is a well-known example of step $r=2$ (stratified) group.
\begin{cor}\label{cor1}
	Let $\mathbb{H}^+ = \{(x_1,x_2,x_3)\in \mathbb{H} \, | \, \, x_3>0 \}$ be a half-space of the Heisenberg group $\mathbb{H}$. Then for any $u \in C^{\infty}_0(\mathbb{H}^+)$ we have 
	\begin{equation}
		\int_{\mathbb{H}^+} |\nabla_{\mathbb{H}} u|^2 dx \geq \int_{\mathbb{H}^+} \frac{|x_1|^2+|x_2|^2}{x_3^2}|u|^2 dx,
	\end{equation}
	where $\nabla_{\mathbb{H}} = \{X_1,X_2\}$ . 
\end{cor}
\begin{proof}[Proof of Corollary \ref{cor1}.]
	Recall that the left-invariant vector fields on the Heisenberg group are generated by the basis
	\begin{align*}
		&X_{1} = \frac{\partial }{\partial x_{1}} + 2x_2\frac{\partial }{\partial x_3}, \\ 
		& X_2 = \frac{\partial }{\partial x_{2}} - 2x_1\frac{\partial }{\partial x_3},  
	\end{align*}
with the commutator
	\begin{equation*}
		[X_1,X_2]= - 4\frac{\partial}{\partial x_3}.
	\end{equation*}
	For $x= (x_1,x_2,x_3)$, choosing $\nu=(0,0,1)$ as the unit vector in the direction of $x_3$ and taking $d=0$ in inequality \eqref{Hardy_1}, we get
	\begin{align*}
		X_1(x) = (1,0,2x_2) \,\, \text{and} \,\, X_2(x) = (0,1,-2x_1), 
	\end{align*} 
	and 
	 \begin{align*}
	 &\langle X_1(x), \nu \rangle = 2x_2, \quad \text{and} \quad \langle X_2(x), \nu \rangle = -2x_1, \\
	 & X_1\langle X_1(x), \nu \rangle = 0, \quad \text{and} \quad X_2 \langle X_2(x), \nu \rangle = 0.
	 \end{align*}
	Therefore, with $\mathcal{W}(x)$ as in \eqref{2.2}, we have
	\begin{equation*}
	\frac{\mathcal{W}(x)^2}{dist(x,\partial \G^+)^2}	 = 4\frac{|x_1|^2+|x_2|^2}{x_3^2}.
	\end{equation*} 
	Substituting these into inequality \eqref{Hardy_1} we arrive at 
		\begin{equation*}
	\int_{\mathbb{H}^+} |\nabla_{\mathbb{H}} u|^2 dx \geq \int_{\mathbb{H}^+} \frac{|x_1|^2+|x_2|^2}{x_3^2}|u|^2 dx,
	\end{equation*}
	taking $\beta = -\frac{1}{2}$.
\end{proof}

Let us present an example for the step $r=3$ (stratified) groups. A well-known stratified group with step three is the Engel group, which can be denoted by $\mathbb{E}$. Topologically $\mathbb{E}$ is $\mathbb{R}^4$ with the group law of $\mathbb{E}$, which is given by 
\begin{equation*}
	x \circ y = (x_1+y_1,x_2+y_2,x_3+y_3+P_1,x_4+y_4+ P_2),
\end{equation*}
where 
\begin{align*}
	P_1 = & \frac{1}{2} (x_1y_2 - x_2y_1),\\
	P_2 = & \frac{1}{2}(x_1y_3-x_3y_1) + \frac{1}{12}(x_1^2y_2 - x_1y_1(x_2+y_2)+x_2y_1^2).
\end{align*}
The left-invariant vector fields of $\mathbb{E}$ are generated by the basis 
\begin{align*}
	&X_1 = \frac{\partial}{\partial x_1} - \frac{x_2}{2} \frac{\partial}{\partial x_3} - \left( \frac{x_3}{2}-\frac{x_1x_2}{12}\right)\frac{\partial}{\partial x_4}, \\
	&X_2 = \frac{\partial}{\partial x_2} + \frac{x_1}{2}\frac{\partial}{\partial x_3} + \frac{x_1^2}{12}\frac{\partial}{\partial x_4},\\
	&X_3 = \frac{\partial}{\partial x_3} + \frac{x_1}{2}\frac{\partial}{\partial x_4}, \\
	&X_4 = \frac{\partial}{\partial x_4}.
	\end{align*}
	
	\begin{cor}\label{corE}
		Let $\mathbb{E}^+ = \{x:=(x_1,x_2,x_3,x_4)\in \mathbb{E} \, | \, \, \langle x,\nu\rangle>0 \}$ be a half-space of the Engel group $\mathbb{E}$. Then for all $\beta \in \mathbb{R}$ and $u \in C^{\infty}_0(\mathbb{E}^+)$ we have 
	\begin{align}\label{E1}
	\int_{\mathbb{E}^+} |\nabla_{\mathbb{E}} u|^2 dx \geq  & C_1(\beta)\int_{\mathbb{E}^+} \frac{\langle X_1(x),\nu \rangle^2+\langle X_2(x),\nu \rangle^2}{dist(x,\partial \mathbb{E}^+)^2} |u|^2 dx \\
	&+ \frac{\beta}{3} \int_{\mathbb{E}^+} \frac{x_2\nu_4}{dist(x,\partial \mathbb{E}^+)} |u|^2 dx, \nonumber
	\end{align}
		where $\nabla_{\mathbb{E}} = \{X_1,X_2\}$, $\nu := (\nu_1,\nu_2,\nu_3,\nu_4)$, and $C_1(\beta) = -(\beta^2 +\beta)$. 
	\end{cor}
	\begin{rem}
		If we take $\nu_4=0$ in \eqref{E1}, then we have the following inequality on $\mathbb{E}$, by taking $\beta=-\frac{1}{2}$,
		\begin{equation*}
			\int_{\mathbb{E}^+} |\nabla_{\mathbb{E}} u|^2 dx \geq  \frac{1}{4}\int_{\mathbb{E}^+} \frac{\langle X_1(x),\nu \rangle^2+\langle X_2(x),\nu \rangle^2}{dist(x,\partial \mathbb{E}^+)^2} |u|^2 dx. 
		\end{equation*}
	\end{rem}
\begin{proof}[Proof of Corollary \ref{corE}]
	As we mentioned, the Engel group has the following basis of the left-invariant vector fields 
	\begin{align*}
	&X_1 = \frac{\partial}{\partial x_1} - \frac{x_2}{2} \frac{\partial}{\partial x_3} - \left( \frac{x_3}{2}-\frac{x_1x_2}{12}\right)\frac{\partial}{\partial x_4}, \\
	&X_2 = \frac{\partial}{\partial x_2} + \frac{x_1}{2}\frac{\partial}{\partial x_3} + \frac{x_1^2}{12}\frac{\partial}{\partial x_4},  
	\end{align*}
	with the following two (non-zero) commutators
	\begin{align*}
	&X_3=[X_1,X_2]= \frac{\partial}{\partial x_3} + \frac{x_1}{2}\frac{\partial}{\partial x_4}, \\
	&X_4=[X_1,X_3] =\frac{\partial}{\partial x_4}.
	\end{align*}
	Thus, we have 
	\begin{align*}
			&X_1(x) = \left(1,0,-\frac{x_2}{2},- \left(\frac{x_3}{2} - \frac{x_1x_2}{12}\right)\right),\\ 
		&X_2(x) = \left(0,1,\frac{x_1}{2},\frac{x_1^2}{12}\right).
	\end{align*}
	A direct calculation gives that
	\begin{align*}
	&\langle X_1(x), \nu \rangle = \nu_1 - \frac{x_2}{2} \nu_3 - \left(\frac{x_3}{2} - \frac{x_1x_2}{12}\right)\nu_4, \\
	&\langle X_2(x), \nu \rangle = \nu_2 + \frac{x_1}{2}\nu_3 + \frac{x_1^2}{12} \nu_4,\\
	& X_1\langle X_1(x), \nu \rangle =
	\frac{x_2}{12}\nu_4 + \frac{x_2}{4}\nu_4=\frac{x_2\nu_4}{3},\\
	& X_2\langle X_2(x), \nu \rangle = 0.
	\end{align*}
	Now substituting these into inequality \eqref{Hardy_1} we obtain the desired result.  
\end{proof}

\subsection{$L^p$-Hardy inequality on $\G^+$} Here we construct an $L^p$ version of the geometric Hardy inequality on the half-space of $\G$ as a generalisation of the previous theorem. We define the $p$-version of the angle function by $\mathcal{W}_p$, which is given by the formula
\begin{equation}
\mathcal{W}_p(x) = \left( \sum_{i=1}^{N} |\langle X_i (x), \nu \rangle|^{p} \right)^{\frac{1}{p}}.
\end{equation}
\begin{thm}\label{thm2}
	Let $\G^+$ be a half-space of a stratified group $\G$. Then for all $\beta \in \mathbb{R}$ we have
\begin{align}\label{eq2}
\int_{\G^+}  \sum_{i=1}^{N}|X_i u|^p dx \geq& C_2(\beta,p) \int_{\G^+} \frac{\mathcal{W}_p(x)^{p}}{dist(x,\partial \G^+)^{p}} |u|^p dx \nonumber\\
& +\beta (p-1) \int_{\G^+} \sum_{i=1}^{N} \left( \frac{|\langle X_i (x), \nu \rangle|}{dist(x,\partial \G^+)} \right)^{p-2}  \frac{X_i\langle X_i(x),\nu\rangle }{dist(x,\partial \G^+)} |u|^p dx
\end{align}
	for all $u \in C^{\infty}_0(\G^+)$, $1<p<\infty$ and $C_2(\beta,p):=-(p-1)( |\beta|^{\frac{p}{p-1}} + \beta)$. 
\end{thm}
\begin{proof}[Proof of Theorem \ref{thm2}] We use the standard method such as the divergence theorem to obtain the inequality \eqref{eq2}. For $W \in C^{\infty}(\G^+)$ and $f \in C^1(\G^+)$, a direct calculation shows that
	\begin{align}\label{proof1}
		\int_{\G^+} {\rm div_{\G}} (fW) |u|^p dx &= - \int_{\G^+} fW \cdot \nabla_{\G} |u|^p dx \nonumber \\ 
		& = -p\int_{\G^+} f \langle W, \nabla_{\G} u\rangle |u|^{p-1}dx \nonumber\\ 
		& \leq p \left(\int_{\G^+} |\langle W,\nabla_{\G} u \rangle|^p dx \right)^{\frac{1}{p}} \left( \int_{\G^+} |f|^{\frac{p}{p-1}}|u|^p dx\right)^{\frac{p-1}{p}}.		
	\end{align}
	Here in the last line H\"older's inequality was applied. For $p>1$ and $ q>1$ with $\frac{1}{p}+\frac{1}{q}=1$ recall Young's inequality 
	\begin{equation*}
		ab\leq \frac{a^p}{p} + \frac{b^q}{q}, \, \text{for} \,\, a \geq 0, \, b\geq 0.
	\end{equation*}
	Let us set that 
	\begin{align*}
		a := \left(\int_{\G^+} |\langle W,\nabla_{\G} u \rangle|^p dx \right)^{\frac{1}{p}} \quad \text{and} \quad b := \left( \int_{\G^+} |f|^{\frac{p}{p-1}}|u|^p dx\right)^{\frac{p-1}{p}}.
	\end{align*}
	By using Young's inequality in \eqref{proof1} and rearranging the terms, we arrive at 
	\begin{equation}\label{prefinal}
	\int_{\G^+} |\langle W,\nabla_{\G} u \rangle|^p dx \geq 
			\int_{\G^+} \left({\rm div_{\G}} (fW)   - (p-1)  |f|^{\frac{p}{p-1}}\right)|u|^p dx.
	\end{equation}
	We choose $W := I_i$, which has the following form $I_i=(\overset{i}{\overbrace{0,\ldots,1}},\ldots,0)$ and set
	\begin{equation*}
		f = \beta \frac{|\langle X_i (x), \nu \rangle|^{p-1}}{dist(x,\partial \G^+)^{p-1}}.
	\end{equation*}
	Now we calculate
	\begin{align*}
		{\rm div_{\G}} (Wf) &=(\nabla_{\G} \cdot I_i)f =  X_i f = \beta X_i \left( \frac{|\langle X_i (x), \nu\rangle|}{dist(x,\partial \G^+)} \right)^{p-1}\\
		& = \beta (p-1) \left( \frac{|\langle X_i (x), \nu \rangle|}{dist(x,\partial \G^+)} \right)^{p-2} X_i \left(\frac{\langle X_i (x), \nu \rangle}{dist(x,\partial \G^+)} \right)  \\
		& =   \beta (p-1) \left( \frac{|\langle X_i (x), \nu \rangle|}{dist(x,\partial \G^+)} \right)^{p-2} \left( \frac{X_i\langle X_i(x),\nu\rangle }{dist(x,\partial \G^+)} -  \frac{|\langle X_i (x), \nu \rangle|^2}{dist(x,\partial \G^+)^2}\right) \\
		& = \beta (p-1) \left[\left( \frac{|\langle X_i (x), \nu \rangle|}{dist(x,\partial \G^+)} \right)^{p-2} \left( \frac{X_i\langle X_i(x),\nu\rangle }{dist(x,\partial \G^+)} \right) - \frac{|\langle X_i (x), \nu \rangle|^{p}}{dist(x,\partial \G^+)^{p}}\right],
	\end{align*}
	and 
	\begin{equation*}
		|f|^{\frac{p}{p-1}} = |\beta|^{\frac{p}{p-1}} \frac{|\langle X_i (x), \nu \rangle|^{p}}{dist(x,\partial \G^+)^{p}}.
	\end{equation*}
	We also have
	\begin{equation*}
		\langle W,\nabla_{\G} u \rangle = \overset{i}{\overbrace{(0,\ldots,1},\ldots,0)} \cdot (X_1u,\ldots, X_i u, \ldots,X_N u)^T = X_i u.
	\end{equation*}
	Inserting the above calculations in \eqref{prefinal} and summing over $i=1,\ldots,N$, we arrive at
	\begin{align}\label{equation}
		\int_{\G^+} \sum_{i=1}^{N}|X_i u|^p dx \geq &-(p-1)( |\beta|^{\frac{p}{p-1}} + \beta) \int_{\G^+}\sum_{i=1}^{N} \frac{|\langle X_i (x), \nu \rangle|^{p}}{dist(x,\partial \G^+)^{p}} |u|^p dx \nonumber\\
		& +\beta (p-1) \int_{\G^+} \sum_{i=1}^{N} \left( \frac{|\langle X_i (x), \nu \rangle|}{dist(x,\partial \G^+)} \right)^{p-2}  \frac{X_i\langle X_i(x),\nu\rangle }{dist(x,\partial \G^+)} |u|^p dx.
	\end{align}
	 We complete the proof of Theorem \ref{thm2}.
\end{proof}
\begin{rem}
	For $p\geq 2$, since
	\begin{equation}
		|\nabla_{\G} u|^p = \left( \sum_{i=1}^{N} |X_i u|^2 \right)^{\frac{p}{2}} \geq \sum_{i=1}^{N}\left(  |X_i u|^2 \right)^{\frac{p}{2}}, 
	\end{equation}
	we have the following inequality 
	\begin{align}
	\int_{\G^+}  |\nabla_{\G} u|^p dx \geq& C_2(\beta,p) \int_{\G^+} \frac{\mathcal{W}_p(x)^{p}}{dist(x,\partial \G^+)^{p}} |u|^p dx \\
	& +\beta (p-1) \int_{\G^+} \sum_{i=1}^{N} \left( \frac{|\langle X_i (x), \nu \rangle|}{dist(x,\partial \G^+)} \right)^{p-2}  \frac{X_i\langle X_i(x),\nu\rangle }{dist(x,\partial \G^+)} |u|^p dx.\nonumber
	\end{align}
\end{rem}

\section{Hardy inequalities on a convex domain of $\G$}\label{convex}
In this section, we present the geometric Hardy inequalities on the convex domains in stratified groups. The convex domain is understood in the sense of the Euclidean space.  Let $\Omega$ be a convex domain of a stratified group $\G$ and let $\partial \Omega$ be its boundary. Below for $x\in\Omega$ we denote by $\nu(x)$ the unit normal for $\partial \Omega$ at a point $\hat{x}\in\partial\Omega$ such that $dist(x,\Omega)=dist(x,\hat{x})$. For the half-plane, we have the distance from the boundary $dist(x,\partial \Omega) = \langle x, \nu\rangle -d$.
As it is introduced in the previous section we also have the generalised angle function
\begin{equation*}
\mathcal{W}_p(x) = \left( \sum_{i=1}^{N} |\langle X_i (x), \nu \rangle|^{p} \right)^{\frac{1}{p}},
\end{equation*}
with 
$\mathcal{W}(x):=\mathcal{W}_2(x)$.
\subsection{Geometric $L^2$-Hardy inequality on a convex domain of $\G$}
\begin{thm}\label{thm_convex_1}
	Let $\Omega$ be a convex domain of a stratified group $\G$. Then for $\beta<0$ we have
	\begin{equation}\label{3.1}
		\int_{\Omega} |\nabla_{\G} u|^2 dx \geq C_1(\beta) \int_{\Omega}  \frac{\mathcal{W}(x)^2}{dist(x,\partial \Omega)^2} |u|^2 dx + \beta \int_{\Omega}\sum_{i=1}^{N} \frac{X_i\langle X_i(x),\nu \rangle}{dist(x,\partial \Omega)} |u|^2 dx	
	\end{equation}
	for all $u \in C_0^{\infty}(\Omega)$, and $C_1(\beta):=-(\beta^2+ \beta)$.
	
\end{thm}
\begin{proof}[Proof of Theorem \ref{thm_convex_1}]
	We follow the approach of Simon Larson \cite{Larson} by proving inequality \eqref{3.1} in the case when $\Omega$ is a convex polytope. We denote its facets by $\{ \mathcal{F}_j \}_j$ and unit normals of these facets by $\{\nu_j\}_j$, which are directed inward. Then $\Omega$ can be constructed by the union of the disjoint sets $\Omega_j := \{x \in \Omega: dist(x, \partial \Omega) = dist(x,\mathcal{F}_j) \}$.
	Now we apply the same method as in the case of the  half-space $\G^+$ for each element $\Omega_j$ with one exception that not all the boundary values are zero when we use the partial integration.   
	As in the previous computation we have
	\begin{align*}
			0\leq \int_{\Omega_j} |\nabla_{\G} u + \beta W u|^2 dx &=
		 \int_{\Omega_j} \sum_{i=1}^{N} |X_i u+ \beta W_i u|^2 dx \\
		& = \int_{\Omega_j}  \sum_{i=1}^{N} \left[ |X_i u|^2 +2{\rm Re} \beta W_i u X_iu + \beta^2 W_i^2 |u|^2 \right] dx \\ 
		& = \int_{\Omega_j}  \sum_{i=1}^{N} \left[ |X_i u|^2 + \beta W_i X_i|u|^2 + \beta^2 W_i^2 |u|^2 \right] dx \\ 
		& = \int_{\Omega_j}  \sum_{i=1}^{N} \left[ |X_i u|^2 - \beta (X_iW_i) |u|^2 + \beta^2 W_i^2 |u|^2 \right] dx \\
		& + \beta \int_{\partial \Omega_j} \sum_{i=1}^{N}   W_i\langle X_i(x), n_j(x)\rangle |u|^2 d\Gamma_{\partial \Omega_j}(x), 
	\end{align*}  
	where $n_j$ is the unit normal of $\partial \Omega_j$  which is directed outward. Since $\mathcal{F}_j \subset \partial \Omega_j$ we have $n_j = -\nu_j$.
	
	The boundary terms on $\partial \Omega$ vanish since $u$ is compactly supported in $\Omega$. So we only deal with the parts of $\partial \Omega_j$ in $\Omega$. Note that for every facet of $\partial \Omega_j$ there exists some $\partial \Omega_l$ which shares this facet. We denote by $\Gamma_{jl}$ the common facet of $\partial \Omega_j$ and $\partial \Omega_l$, with $n_k|_{\Gamma_{jl}}= -n_l|_{\Gamma_{jl}}$. 
	From the above expression we get the following inequality
	\begin{align}\label{eq3}
	\int_{\Omega_j} |\nabla_{\G} u|^2 dx \geq & \int_{\Omega_j} \sum_{i=1}^{N} \left[  (\beta (X_iW_i)  - \beta^2 W_i^2 )|u|^2 \right] dx \\
	& - \beta \int_{\partial \Omega_j} \sum_{i=1}^{N}   W_i\langle X_i(x), n_j(x)\rangle |u|^2 d\Gamma_{\partial \Omega_j}(x). \nonumber	
	\end{align}
	Now we choose $W_i$ in the form
	\begin{equation*}
	W_i (x) = \frac{\langle X_i(x),\nu_j\rangle}{dist(x,\partial \Omega_j)} = \frac{\langle X_i(x),\nu_j\rangle}{\langle x,\nu_j\rangle - d},
	\end{equation*}
	and a direct computation shows that
	\begin{equation}\label{2}
	X_i W_i(x) =  \frac{X_i\langle X_i(x),\nu_j \rangle}{dist(x,\partial \Omega_j)}- \frac{\langle X_i(x),\nu_j \rangle^2}{dist(x,\partial \Omega_j)^2}.
	\end{equation}
	Inserting the expression \eqref{2} into inequality \eqref{eq3} we get 
	\begin{align}
	\int_{\Omega_j} &|\nabla_{\G} u|^2 dx \geq -(\beta^2+ \beta)\int_{\Omega_j} \sum_{i=1}^{N} \frac{\langle X_i(x),\nu_j \rangle^2}{dist(x,\partial \Omega_j)^2} |u|^2 dx \\
	&+ \beta \int_{\Omega_j}\sum_{i=1}^{N} \frac{X_i\langle X_i(x),\nu_j \rangle}{dist(x,\partial \Omega_j)} |u|^2 dx
	 - \beta \int_{\Gamma_{jl}} \sum_{i=1}^{N} \frac{\langle X_i(x),\nu_j\rangle \langle X_i(x), n_{jl} \rangle}{dist(x,\mathcal{F}_j)} |u|^2  d\Gamma_{jl}. \nonumber
	\end{align}
 Now we sum over all partition elements $\Omega_j$ and let $n_{jl}=n_k|_{{\Gamma_{jl}}}$, i.e. the unit normal of $\Gamma_{jl}$ pointing from $\Omega_j$ into $\Omega_l$. Then we get 
	\begin{align*}
		\int_{\Omega} |\nabla_{\G} u|^2 dx \geq & -(\beta^2+ \beta) \int_{\Omega} \sum_{i=1}^{N} \frac{\langle X_i(x),\nu \rangle^2}{dist(x,\partial \Omega)^2} |u|^2 dx \\
		& + \beta \int_{\Omega}\sum_{i=1}^{N} \frac{X_i\langle X_i(x),\nu \rangle}{dist(x,\partial \Omega)} |u|^2 dx\\
		& - \beta \sum_{j\neq l} \int_{\Gamma_{jl}} \sum_{i=1}^{N} \frac{\langle X_i(x),\nu_j\rangle \langle X_i(x), n_{jl} \rangle}{dist(x,\mathcal{F}_j)} |u|^2  d\Gamma_{jl} \\
		 = & -(\beta^2+ \beta) \int_{\Omega} \sum_{i=1}^{N} \frac{\langle X_i(x),\nu \rangle^2}{dist(x,\partial \Omega)^2} |u|^2 dx \\
		& + \beta \int_{\Omega}\sum_{i=1}^{N} \frac{X_i\langle X_i(x),\nu \rangle}{dist(x,\partial \Omega)} |u|^2 dx\\
		& -\beta \sum_{j<l} \int_{\Gamma_{jl}} \sum_{i=1}^{N} \frac{\langle X_i(x),\nu_j-\nu_l\rangle \langle X_i(x), n_{jl} \rangle}{dist(x,\mathcal{F}_j)} |u|^2  d\Gamma_{jl}.
	\end{align*}   
	Here we used the fact that (by the definition) $\Gamma_{jl}$ is a set with $dist(x, \mathcal{F}_j)=dist(x, \mathcal{F}_l)$. 
	From
	\begin{equation*}
		\Gamma_{jl} = \{ x: x\cdot \nu_j - d_j = x \cdot \nu_l -d_l \}
	\end{equation*}
	rearranging $ x \cdot (\nu_j - \nu_l)- d_j + d_l=0$ we see that $\Gamma_{jl}$ is a hyperplane with a normal $\nu_j - \nu_l$. Thus, $\nu_j-\nu_l$ is parallel to $n_{jl}$ and one only needs to check that $(\nu_j-\nu_l)\cdot n_{jl}>0$. Observe that $n_{jl}$ points out and $\nu_j$ points into $j$-th partition element, so $\nu_j \cdot n_{jl}$ is non-negative. Similarly, we see that $\nu_l \cdot n_{jl}$ is non-positive. This means we have $(\nu_j-\nu_l)\cdot n_{jl}>0$. In addition, it is easy to see that
	\begin{align*}
		|\nu_j - \nu_l|^2 = (\nu_j - \nu_l)\cdot (\nu_j - \nu_l)&= 2 - 2 \nu_j \cdot \nu_l \\
		& = 2 -2 \cos (\alpha_{jl}),
	\end{align*} 
	which implies that
	$$(\nu_j - \nu_l)\cdot n_{jl}= \sqrt{2-2\cos(\alpha_{jl})},$$
	where $\alpha_{jl}$ is the angle between $\nu_j$ and $\nu_l$. So we obtain 
	\begin{align*}
		\int_{\Omega} |\nabla_{\G} u|^2 dx \geq&  -(\beta^2+ \beta) \int_{\Omega} \sum_{i=1}^{N} \frac{\langle X_i(x),\nu \rangle^2}{dist(x,\partial \Omega)^2} |u|^2 dx \\
		& + \beta \int_{\Omega}\sum_{i=1}^{N} \frac{X_i\langle X_i(x),\nu \rangle}{dist(x,\partial \Omega)} |u|^2 dx\\
		& -\beta \sum_{j<l} \sum_{i=1}^{N} \int_{\Gamma_{jl}} \sqrt{1- \cos (\alpha_{jl})} \frac{\langle X_i(x),n_{jl}\rangle^2}{dist(x,\mathcal{F}_j)}|u|^2 d \Gamma_{jl}.
	\end{align*}
	Here with $\beta<0$
	and due to the boundary term signs we verify the inequality for the polytope convex domains.

	Let us now consider the general case, that is, when $\Omega$ is an arbitrary convex domain. For each $u \in C_0^{\infty}(\Omega)$ one can always choose an increasing sequence of convex polytopes $\{\Omega_j\}_{j=1}^{\infty}$ such that $u \in C_0^{\infty}(\Omega_1),\; \Omega_j \subset \Omega$ and $\Omega_j \rightarrow \Omega$ as $j \rightarrow \infty$. Assume that $\nu_j(x)$ is the above map $\nu $  (corresponding to $\Omega_j$) we compute  
	\begin{align*}
		\int_{\Omega} |\nabla_{\G} u|^2 dx =& \int_{\Omega_j} |\nabla_{\G} u|^2 dx \\
		\geq & 
		 -(\beta^2+ \beta) \int_{\Omega_j} \sum_{i=1}^{N} \frac{\langle X_i(x),\nu_j \rangle^2}{dist(x,\partial \Omega_j)^2} |u|^2 dx \\
		& + \beta \int_{\Omega_j}\sum_{i=1}^{N} \frac{X_i\langle X_i(x),\nu_j \rangle}{dist(x,\partial \Omega_j)} |u|^2 dx\\
		 =& -(\beta^2+ \beta) \int_{\Omega} \sum_{i=1}^{N} \frac{\langle X_i(x),\nu_j \rangle^2}{dist(x,\partial \Omega_j)^2} |u|^2 dx \\
		 & + \beta \int_{\Omega}\sum_{i=1}^{N} \frac{X_i\langle X_i(x),\nu_j \rangle}{dist(x,\partial \Omega_j)} |u|^2 dx\\
		 \geq & -(\beta^2+ \beta) \int_{\Omega} \sum_{i=1}^{N} \frac{\langle X_i(x),\nu_j \rangle^2}{dist(x,\partial \Omega)^2} |u|^2 dx \\
		 & + \beta \int_{\Omega}\sum_{i=1}^{N} \frac{X_i\langle X_i(x),\nu_j \rangle}{dist(x,\partial \Omega)} |u|^2 dx		
	\end{align*}
Now we obtain the desired result when $j \rightarrow \infty$.
\end{proof}

\subsection{$L^p$-Hardy's inequality on a convex domain of $\G$}
In this section we give the $L^p$-version of the previous results.
 \begin{thm}\label{thm4}
 Let $\Omega$ be a convex domain of a stratified group $\G$. Then for $\beta<0$ we have
\begin{align}\label{eq_conv}
\int_{\Omega} \sum_{i=1}^{N} |X_i u|^p  dx \geq& C_2(\beta,p)
\int_{\Omega} \frac{\mathcal{W}_p(x)^{p}}{dist(x,\partial \Omega)^{p}} |u|^p dx\\ \nonumber
&+ \beta (p-1) \int_{\Omega} \sum_{i=1}^{N}\left( \frac{|\langle X_i (x), \nu \rangle|}{dist(x,\partial \Omega)} \right)^{p-2} \left( \frac{X_i\langle X_i(x),\nu\rangle }{dist(x,\partial \Omega)} \right) |u|^p dx,
\end{align}
 for all $u \in C_0^{\infty}(\Omega)$, and $C_2(\beta,p):=-(p-1)( |\beta|^{\frac{p}{p-1}} + \beta)$.	
 \end{thm}
\begin{proof}[Proof of Theorem \ref{thm4}]
	Let us assume that $\Omega$ is the convex polytope as in the $p=2$ case. Thus, we consider the partition $\Omega_j$ as the previous case. For $f \in C^1(\Omega_j)$ and $W \in C^{\infty}(\Omega_j)$,  a simple calculation shows that
	\begin{align}\label{proof2}
		\int_{\Omega_j} {\rm div_{\G}} (fW) |u|^p dx 
 =& -p	\int_{\Omega_j} f \langle W, \nabla_{\G} u\rangle |u|^{p-1}dx + \int_{\partial \Omega_j} f \langle W, n_j(x) \rangle |u|^p d\Gamma_{\partial \Omega_j}(x) \nonumber\\ 
	 \leq& p \left(\int_{\Omega} |\langle W,\nabla_{\G} u \rangle|^p dx \right)^{\frac{1}{p}} \left( 	\int_{\Omega_j} |f|^{\frac{p}{p-1}}|u|^p dx\right)^{\frac{p-1}{p}} \\
	& + \int_{\partial \Omega_j} f \langle W, n_j(x) \rangle |u|^p d\Gamma_{\partial \Omega_j}(x) \nonumber.		
	\end{align}
	In the last line H\"older's inequality was applied. Recall again Young's inequality for $p>1$, $q>1$ and $\frac{1}{p}+\frac{1}{q}=1$, we have 
$
	ab\leq \frac{a^p}{p} + \frac{b^q}{q}, \,\,\, \text{for} \,\, a \geq 0, \, b\geq 0.
$
	We now take $q := \frac{p}{p-1}$ and
	\begin{align*}
	a := \left(\int_{\Omega} |\langle W,\nabla_{\G} u \rangle|^p dx \right)^{\frac{1}{p}} \quad \text{and} \quad b := \left( \int_{\Omega} |f|^{\frac{p}{p-1}}|u|^p dx\right)^{\frac{p-1}{p}}.
	\end{align*}
	By using Young's inequality in \eqref{proof2} and rearranging the terms, we arrive at 
	\begin{align}\label{prefinal1}
	\int_{\Omega_j} |\langle W,\nabla_{\G} u \rangle|^p dx &\geq \int_{\Omega} \left({\rm div_{\G}} (fW)   - (p-1)  |f|^{\frac{p}{p-1}}\right)|u|^p dx \\
	&- \int_{\partial \Omega_j} f \langle W, n_j(x) \rangle |u|^p d\Gamma_{\partial \Omega_j}(x) \nonumber.
	\end{align}
	We choose $W := I_i$ as a unit vector of the $i^{th}$ component and let
	\begin{equation*}
	f = \beta \frac{|\langle X_i (x), \nu_j\rangle|^{p-1}}{dist(x,\mathcal{F}_j)^{p-1}}.
	\end{equation*}
	As before a direct calculation shows that 
\begin{align*}
{\rm div_{\G}} (Wf) &=  X_i f = \beta X_i \left( \frac{|\langle X_i (x), \nu_j \rangle|}{dist(x,\partial \mathcal{F}_j)} \right)^{p-1}\\
&= \beta (p-1) \left( \frac{|\langle X_i (x), \nu_j \rangle|}{dist(x,\partial \mathcal{F}_j)} \right)^{p-2} X_i \left(\frac{\langle X_i (x), \nu_j \rangle}{dist(x,\partial \mathcal{F}_j)} \right)  \\
& =   \beta (p-1) \left( \frac{|\langle X_i (x), \nu_j \rangle|}{dist(x,\partial \mathcal{F}_j)} \right)^{p-2} \left( \frac{X_i\langle X_i(x),\nu_j\rangle }{dist(x,\partial \mathcal{F}_j)} -  \frac{|\langle X_i (x), \nu_j \rangle|^2}{dist(x,\partial \mathcal{F}_j)^2}\right) \\
& = \beta (p-1) \left[ \left( \frac{|\langle X_i (x), \nu_j \rangle|}{dist(x,\partial \mathcal{F}_j)} \right)^{p-2} \left( \frac{X_i\langle X_i(x),\nu_j\rangle }{dist(x,\partial \mathcal{F}_j)} \right) -  \frac{|\langle X_i (x), \nu_j \rangle|^{p}}{dist(x,\partial \mathcal{F}_j)^{p}}\right],
\end{align*}
and 
\begin{equation*}
|f|^{\frac{p}{p-1}} = |\beta|^{\frac{p}{p-1}} \frac{|\langle X_i (x), \nu_j\rangle|^{p}}{dist(x,\mathcal{F}_j)^{p}}.
\end{equation*}
We also have
\begin{equation*}
\langle W,\nabla_{\G} u \rangle = \overset{i}{(\overbrace{0,\ldots,1},\ldots,0)} \cdot (X_1u,\ldots, X_i u, \ldots,X_N u)^T = X_i u.
\end{equation*}
Inserting the above calculations into \eqref{prefinal1} and summing over $i=\overline{1,N}$, we arrive at
\begin{align}\label{equation1}
\int_{\Omega_j} \sum_{i=1}^{N}|X_i u|^p dx \geq& -(p-1)( |\beta|^{\frac{p}{p-1}} + \beta) \int_{\Omega_j} \sum_{i=1}^{N}\frac{|\langle X_i (x), \nu_j\rangle|^{p}}{dist(x,\partial\mathcal{F}_j)^{p}} |u|^p dx \\
&+ \beta(p-1)  \int_{\Omega_j} \sum_{i=1}^{N}\left( \frac{|\langle X_i (x), \nu_j \rangle|}{dist(x,\partial \mathcal{F}_j)} \right)^{p-2} \left( \frac{X_i\langle X_i(x),\nu_j\rangle }{dist(x,\partial \mathcal{F}_j)} \right) |u|^p dx \nonumber  \\
&-\beta\int_{\partial \Omega_j} \sum_{i=1}^{N}\left(\frac{|\langle X_i(x),\nu_j\rangle|}{dist(x,\mathcal{F}_j)}\right)^{p-1} \langle X_i(x), n_j(x) \rangle |u|^p d\Gamma_{\partial \Omega_j}(x). \nonumber
\end{align}

Now summing up over $\Omega_j$, and with the interior boundary terms we have
\begin{align*}\label{equa}
\int_{\Omega} \sum_{i=1}^{N} |X_i u|^p  dx \geq& -(p-1)( |\beta|^{\frac{p}{p-1}} + \beta)
 \sum_{i=1}^{N}\int_{\Omega} \frac{|\langle X_i (x), \nu \rangle|^{p}}{dist(x,\partial \Omega)^{p}} |u|^p dx \\
& + \beta(p-1) \sum_{i=1}^{N} \int_{\Omega} \left( \frac{|\langle X_i (x), \nu \rangle|}{dist(x,\partial \Omega)} \right)^{p-2} \left( \frac{X_i\langle X_i(x),\nu\rangle }{dist(x,\partial \Omega)} \right) |u|^p dx\\ 
&-\beta \sum_{j\neq l} \sum_{i=1}^{N}\int_{\Gamma_{jl}} \left(\frac{|\langle X_i(x),\nu_j\rangle|}{dist(x,\mathcal{F}_j)}\right)^{p-1} \langle X_i(x), n_{jl}(x) \rangle |u|^p d\Gamma_{jl} \\
=& -(p-1)( |\beta|^{\frac{p}{p-1}} + \beta)
\sum_{i=1}^{N}\int_{\Omega} \frac{|\langle X_i (x), \nu \rangle|^{p}}{dist(x,\partial \Omega)^{p}} |u|^p dx\\
&+ \beta(p-1) \sum_{i=1}^{N} \int_{\Omega} \left( \frac{|\langle X_i (x), \nu \rangle|}{dist(x,\partial \Omega)} \right)^{p-2} \left( \frac{X_i\langle X_i(x),\nu\rangle }{dist(x,\partial \Omega)} \right) |u|^p dx\\ 
&-\beta \sum_{j< l} \sum_{i=1}^{N}\int_{\Gamma_{jl}} \left[ \left(\frac{|\langle X_i(x),\nu_j\rangle|}{dist(x,\mathcal{F}_j)}\right)^{p-1} \langle X_i(x), n_{jl}(x) \rangle \right.\\
&-  \left.\left(\frac{|\langle X_i(x),\nu_l\rangle|}{dist(x,\mathcal{F}_l)}\right)^{p-1} \langle X_i(x), n_{jl}(x) \rangle \right] |u|^p d\Gamma_{jl}
\end{align*}
As in the earlier case if the boundary term is positive we can discard it, so we want to show that
\begin{equation*}
	\left[ \left(\frac{|\langle X_i(x),\nu_j\rangle|}{dist(x,\mathcal{F}_j)}\right)^{p-1} \langle X_i(x), n_{jl}(x) \rangle - \left(\frac{|\langle X_i(x),\nu_l\rangle|}{dist(x,\mathcal{F}_l)}\right)^{p-1} \langle X_i(x), n_{jl}(x) \rangle \right] \geq 0. 
\end{equation*}
Noting the fact that $ n_{jl} = \frac{\nu_j - \nu_l}{\sqrt{2-2\cos(\alpha_{jl})}}$ and $dist(x,\mathcal{F}_j)=dist(x,\mathcal{F}_l)$ on $\Gamma_{jl}$, we arrive at
\begin{align*}
& \frac{1}{2-2\cos(\alpha_{jl})}\left[ \left(\frac{|\langle X_i(x),\nu_j\rangle|}{dist(x,\mathcal{F}_j)}\right)^{p-1} \langle X_i(x),\nu_j - \nu_l \rangle - \left(\frac{|\langle X_i(x),\nu_l\rangle|}{dist(x,\mathcal{F}_l)}\right)^{p-1} \langle X_i(x), \nu_j - \nu_l \rangle \right]\\
&=\frac{ |\langle X_i(x), \nu_j \rangle|^p - |\langle X_i(x), \nu_j \rangle|^{p-1} \langle X_i(x), \nu_l \rangle - |\langle X_i(x), \nu_l \rangle|^{p-1} \langle X_i(x), \nu_j \rangle  +|\langle X_i(x), \nu_l \rangle|^p }{(2-2\cos(\alpha_{jl}))dist(x,\mathcal{F}_j)^{p-1}} \\ 
&=	\frac{ \left( |\langle X_i(x), \nu_j \rangle| -|\langle X_i(x), \nu_l \rangle| \right) \left( |\langle X_i(x), \nu_j \rangle|^{p-1} -|\langle X_i(x), \nu_l \rangle|^{p-1}\right)}{(2-2\cos(\alpha_{jl}))dist(x,\mathcal{F}_j)^{p-1}} \geq 0.
\end{align*}
Here we have used the equality $(a-b)(a^{p-1}-b^{p-1})=a^p-a^{p-1}b-b^{p-1}a+b^{p-1}$ with $a=|\langle X_i(x), \nu_j \rangle|$ and $b=|\langle X_i(x), \nu_l \rangle|$.
From the above expression we note that the boundary term in $\Omega$ is positive and $\beta <0$. By discarding the boundary term we complete the proof.
\end{proof}
\begin{rem}
	For $p\geq 2$, since
	\begin{equation}
	|\nabla_{\G} u|^p = \left( \sum_{i=1}^{N} |X_i u|^2 \right)^{\frac{p}{2}} \geq \sum_{i=1}^{N}\left(  |X_i u|^2 \right)^{\frac{p}{2}}, 
	\end{equation}
	we have the following inequality 
	\begin{align}
	\int_{\Omega} |\nabla_{\G} u|^p  dx \geq& C_2(\beta,p)
	\int_{\Omega} \frac{\mathcal{W}_p(x)^{p}}{dist(x,\partial \Omega)^{p}} |u|^p dx\\ \nonumber
	&+ \beta (p-1) \int_{\Omega} \sum_{i=1}^{N}\left( \frac{|\langle X_i (x), \nu \rangle|}{dist(x,\partial \Omega)} \right)^{p-2} \left( \frac{X_i\langle X_i(x),\nu\rangle }{dist(x,\partial \Omega)} \right) |u|^p dx.
	\end{align}
\end{rem}

\end{document}